\documentclass{amsart}
\newtheorem{theorem}{Theorem}[section]
\newtheorem{lemma}[theorem]{Lemma}
\newtheorem{cor}[theorem]{Corollary}

\theoremstyle{definition}
\newtheorem{definition}[theorem]{Definition}
\newtheorem{example}[theorem]{Example}

\newtheorem{rem}[theorem]{Remark}

\theoremstyle{remark}

\numberwithin{equation}{section}

\begin{document}

\title[Closure of the cone of sums of $2d$-powers]{Closure of the cone of sums of $2d$-powers in certain weighted $\ell_1$-seminorm topologies}

\author{Mehdi Ghasemi, Murray Marshall, Sven Wagner}

\address{Department of Mathematics and Statistics,
University of Saskatchewan, \newline \indent
Saskatoon, SK S7N 5E6, Canada}
\email{mehdi.ghasemi@usask.ca, marshall@math.usask.ca, wagner@math.usask.ca}
\date{July 19, 2011}

\begin{abstract} In \cite{BC} Berg, Christensen and Ressel prove that the closure of the cone of sums of squares $\sum \mathbb{R}[\underline{X}]^2$ in the polynomial ring $\mathbb{R}[\underline{X}] := \mathbb{R}[X_1,\dots,X_n]$ in the topology induced by the $\ell_1$-norm is equal to $\operatorname{Pos}([-1,1]^n)$, the cone consisting of all polynomials which are non-negative on the hypercube $[-1,1]^n$. The result is deduced as a corollary of a general result, also established in \cite{BC}, which is valid for any commutative semigroup. 
In later work Berg and Maserick \cite{BM} and Berg, Christensen and Ressel \cite{BCR}
establish an even more general result, for a commutative semigroup with involution, for the closure of the cone of sums of squares of symmetric elements in the weighted $\ell_1$-seminorm topology associated to an absolute value.
In the present paper we give a new proof of these results which is based on Jacobi's representation theorem \cite{J}. At the same time, we use Jacobi's representation theorem to extend these results from sums of squares to sums of $2d$-powers, proving, in particular, that for any integer $d\ge 1$, the closure of the cone of sums of $2d$-powers $\sum \mathbb{R}[\underline{X}]^{2d}$ in $\mathbb{R}[\underline{X}]$ in the topology induced by the $\ell_1$-norm is equal to $\operatorname{Pos}([-1,1]^n)$.
\end{abstract}

\maketitle

\section{introduction}

We denote the polynomial ring $\mathbb{R}[X_1,\dots,X_n]$ by $\mathbb{R}[\underline{X}]$ for short.
It was known already by Hilbert \cite{H} that for $n\ge 2$ there are polynomials in $\mathbb{R}[\underline{X}]$ which are non-negative on all of $\mathbb{R}^n$ but are not in the cone $\sum \mathbb{R}[\underline{X}]^2$ consisting of sums of squares. The first explicit example was given by Motzkin \cite{Mot}. Today, many examples are known, e.g., see \cite{Bl}. In \cite{BCJ} Berg, Christensen and Jensen prove that, in the finest locally convex topology, $\sum \mathbb{R}[\underline{X}]^2$ is closed in $\mathbb{R}[\underline{X}]$; also see \cite{S}.

In marked contrast to this result, in \cite{BC} Berg, Christensen and Ressel show that the closure of $\sum \mathbb{R}[\underline{X}]^2$ in $\mathbb{R}[\underline{X}]$ in the topology induced by the $\ell_1$-norm is equal to $\operatorname{Pos}([-1,1]^n)$, the set of all polynomials in $\mathbb{R}[\underline{X}]$ which are non-negative on $[-1,1]^n$. In \cite{BC} the aforementioned result is established in the general context of commutative semigroups. In \cite{BCR} and \cite{BM} the results in \cite{BC} are extended further, to include commutative semigroups with involution and topologies induced by absolute values.

Let $d$ be a positive integer and let $M \subseteq \mathbb{R}[\underline{X}]$ be a $\sum \mathbb{R}[\underline{X}]^{2d}$-module which is archimedean.
In \cite{J} Jacobi proves that any $f\in \mathbb{R}[\underline{X}]$ which is strictly positive on $$K_M:=\{ \underline{x} \in \mathbb{R}^n \mid g(\underline{x})\ge 0\ \forall \ g\in M\}$$ belongs to $M$. Actually, Jacobi proves a more general version of this result which is valid for any commutative ring $A$ with $1$; see Section \ref{background}. There is special interest in the case where $d=1$, $A = \mathbb{R}[\underline{X}]$ and $M$ is finitely generated, because of the application to polynomial optimization in this case;
see \cite{L}. Jacobi's result in this case can be seen as a consequence of Putinar's criterion in \cite{P}.

In the present paper we use Jacobi's theorem to give a new proof of the result of Berg, Christensen and Ressel in \cite{BC} referred to above. At the same time we use Jacobi's theorem to extend this result, proving, for any integer $d\ge 1$, that the closure of $\sum \mathbb{R}[\underline{X}]^{2d}$ in $\mathbb{R}[\underline{X}]$ in the topology induced by the $\ell_1$-norm is equal to $\operatorname{Pos}([-1,1]^n)$. As in \cite{BC}, \cite{BCR} and \cite{BM}, extensions of this result to
absolute values on commutative semigroups with involution are also developed.

In Section \ref{background} we provide necessary background.
In Sections \ref{main case}, \ref{extension} and \ref{weak absolute values} we explain how Jacobi's result can be exploited to prove the aforementioned results of Berg et al and also how it can be used to generalize these results, replacing $2$ by $2d$. Special attention is paid to the polynomial case, i.e., the case where the semigroup in question is $(\mathbb{N}^n,+)$; see Section \ref{main case}. In Section \ref{appendix}, we explain how the simple proof of Jacobi's result in the case $d=1$ given in \cite[Theorem 5.4.4]{M} can be extended to the case $d>1$.

\section{background}\label{background}

Let $A$ be a commutative ring with $1$. For simplicity assume $\mathbb{Q}\subseteq A$. Denote by $X_A$ the set of all (unitary) ring homomorphisms $\alpha : A \rightarrow \mathbb{R}$. For $a\in A$, define $\hat{a} : X_A \rightarrow \mathbb{R}$ by $\hat{a}(\alpha) = \alpha(a)$. Give $X_A$ the weakest topology making each $\hat{a}$, $a\in A$ continuous. We have a ring homomorphism $\hat{} : A \rightarrow \operatorname{Cont}(X_A,\mathbb{R})$ defined by $a \mapsto \hat{a}$.\footnote{ We will abuse the notation occasionally, denoting $\hat{a}(\alpha)$ by $a(\alpha)$ and $\hat{a}$ by $a$.}

\begin{example} If $A = \mathbb{R}[\underline{X}]$, $X_A$ is naturally identified with $\mathbb{R}^n$ via $\alpha \leftrightarrow \underline{x} :=(\alpha(X_1),\dots,\alpha(X_n))$, and $\hat{a}(\alpha) = a(\underline{x})$, i.e., $\hat{a}$ is the polynomial function on $\mathbb{R}^n$ associated to the polynomial $a$. One uses here the fact that the only ring homomorphism from $\mathbb{R}$ to $\mathbb{R}$ is the identity map; see \cite[Proposition 5.4.5]{M}.
\end{example}

We will be interested in the map $a \mapsto \hat{a}|_{\mathcal{K}}$ from $A$ to $\operatorname{Cont}(\mathcal{K},\mathbb{R})$, where $\mathcal{K}$ is a subset of $X_A$. We record the following result:

\begin{theorem}\label{haviland,etc} Suppose $A$ is an $\mathbb{R}$-algebra and $\mathcal{K}$ is a compact subset of $X_A$. Then
\begin{enumerate}
\item The image of $A$ in $\operatorname{Cont}(\mathcal{K},\mathbb{R})$ is dense in the topology induced by the sup norm $\| \phi\| := \sup\{ |\phi(\alpha)| \mid \alpha \in \mathcal{K}\}$.
\item If $L : A \rightarrow \mathbb{R}$ is an $\mathbb{R}$-linear map satisfying $L(\operatorname{Pos}(\mathcal{K})) \subseteq \mathbb{R}_{\ge 0}$ then there exists a unique positive borel measure $\mu$ on $\mathcal{K}$ such that $\forall$ $a\in A$, $L(a)= \int \hat{a} d\mu$.
\end{enumerate}
\end{theorem}

Here, $\operatorname{Pos}(\mathcal{K}) := \{ a \in A \mid \hat{a} \ge 0 \text{ on } \mathcal{K}\}$.

\begin{proof}
(1) This is immediate from the Stone-Weierstrass Approximation Theorem.
(2) $L$ vanishes on the kernel of $\hat{}\,|_{\mathcal{K}}$ so $L$ induces an $\mathbb{R}$-linear map $L' : A' \rightarrow \mathbb{R}$, where $A'$ is the image of $A$ under $\hat{}\,|_{\mathcal{K}}$. An application of the Hahn-Banach Theorem shows that $L'$ extends to a positive $\mathbb{R}$-linear map $L'' : \operatorname{Cont}(\mathcal{K}, \mathbb{R}) \rightarrow \mathbb{R}$. The density of $A'$ in $\operatorname{Cont}(\mathcal{K},\mathbb{R})$ implies the extension $L''$ of $L'$ is unique. The existence and uniqueness of $\mu$ now follows, using the Riesz Representation Theorem.
\end{proof}

\begin{rem} (i) Theorem \ref{haviland,etc}(2) is well-known. It is implicit for example in the proof of \cite[Th\'eor\`eme 14]{Kr}. In the special case $A=\mathbb{R}[\underline{X}]$ it is a consequence of Haviland's Theorem; see \cite{Hav1} \cite{Hav2}.  See \cite[Theorem 3.2.2]{M} for a general result which includes Theorem \ref{haviland,etc}(2) and Haviland's Theorem as special cases.
(ii) Observe that the converse of Theorem \ref{haviland,etc}(2) holds trivially: If $L(a) = \int \hat{a} \, d\mu$ for all $a\in A$, where $\mu$ is a positive borel measure on $\mathcal{K}$, then $L(a)\ge 0$ for all $a \in \operatorname{Pos}(\mathcal{K})$.
\end{rem}

We recall some basic terminology. A \it preprime \rm of $A$ is a subset $P$ of $A$ satisfying $$P+P\subseteq P, \ P\cdot P \subseteq P, \text{ and } \mathbb{Q}_{\ge 0} \subseteq P.$$ $P$ is said to be \it generating \rm if $P-P=A$. If there exists a positive integer $d$ such that $a^{2d}\in P$ for all $a\in A$, then $P$ is called a \it preordering, \rm more precisely, a \it preordering of exponent $2d$.\footnote{Preorderings of odd exponent are not interesting. If $P$ is a preordering of odd exponent then $P=P-P = A$.} \rm Denote by $\sum A^{2d}$ the set of all finite sums of $2d$-powers of elements of $A$.  $\sum A^{2d}$ is the unique smallest preordering of $A$ of exponent $2d$. The polynomial identity $$n!X = \sum_{h=0}^{n-1} (-1)^{n-1-h} \begin{pmatrix} n-1 \\ h\end{pmatrix} [(X+h)^n-h^n],$$ see \cite[Page 325]{HW}, applied with $n=2d$,  shows that any preordering is generating.

A subset $M$ of $A$ is called a \it $P$-module \rm if $$M+M \subseteq M, \ P\cdot M\subseteq M, \text{ and } 1\in M.$$
$M$ is said to be \it archimedean \rm if $\forall$ $a\in A$ $\exists$ $n\in \mathbb{N}$ such that $n+a\in M$.
The \it non-negativity set of $M$ in $X_A$ \rm is the subset $\mathcal{K}_M$ of $X_A$ defined by $$\mathcal{K}_M := \{ \alpha \in X_A \mid \alpha(M) \subseteq \mathbb{R}_{\ge 0}\} = \{ \alpha \in X_A \mid \hat{a}\ge 0 \text{ at } \alpha \ \forall a \in M\}.$$

\begin{rem} If $M$ is archimedean then $\mathcal{K}_M$ is compact. This is well-known. For each $a\in A$ $\exists$ $n_a\in \mathbb{N}$ such that $n_a\pm a \in M$. The map $\alpha \mapsto (\alpha(a))_{a\in A}$ identifies $\mathcal{K}_M$ with a closed subset of the compact space $\prod_{a\in A} [-n_a,n_a]$.
\end{rem}

\begin{theorem}[Jacobi] \label{jacobi theorem} Suppose $M \subseteq A$ is an archimedean $\sum A^{2d}$-module of $A$ for some integer $d\ge 1$. Then, $\forall$ $a\in A$, $$\hat{a}>0 \text{ on } \mathcal{K}_M \ \Rightarrow \ a\in M.$$
\end{theorem}

\begin{proof} See \cite[Theorem 4]{J}.
\end{proof}

See \cite[Hauptsatz]{BS} and \cite[Th\'eor\`eme 12]{Kr} for early variants of Theorem \ref{jacobi theorem}. See \cite[Theorem 6.2]{BSS} and \cite[Theorem 5.4.4]{M}  for other proofs of Theorem \ref{jacobi theorem} in the case $d=1$. See Section \ref{appendix} for the extension of the proof in \cite[Theorem 5.4.4]{M} to the case $d>1$. See \cite[Theorem 2.3]{M0} for an extension of Theorem \ref{jacobi theorem}.

\begin{cor} \label{cor of jacobi theorem} Suppose $A$ is an $\mathbb{R}$-algebra and $M \subseteq A$ is an archimedean $\sum A^{2d}$-module of $A$ for some integer $d\ge 1$. If $L : A \rightarrow \mathbb{R}$ is an $\mathbb{R}$-linear map satisfying $L(M) \subseteq \mathbb{R}_{\ge 0}$ then there exists a unique positive borel measure $\mu$ on $\mathcal{K}_M$ such that $\forall$ $a\in A$ $L(a)= \int \hat{a} d\mu$.
\end{cor}

\begin{proof} Suppose $a\in A$, $a\ge 0$ on $\mathcal{K}_M$, and $\epsilon\in \mathbb{R}$, $\epsilon>0$. Then $a+\epsilon >0$ on $\mathcal{K}_M$ so, by Jacobi's theorem, $a+\epsilon \in M$. Then $L(a+\epsilon)= L(a)+\epsilon L(1) \ge 0$. Since $\epsilon>0$ is arbitrary, this implies $L(a)\ge 0$. This proves $L(\operatorname{Pos}(\mathcal{K}_M)) \subseteq \mathbb{R}_{\ge 0}$, so the result follows now, by Theorem \ref{haviland,etc}(2).
\end{proof}

By a \it topological $\mathbb{R}$-vector space \rm we mean an $\mathbb{R}$-vector space $V$ equipped with a topology such that the addition and scalar multiplcation are continuous. There is no requirement that the topology be Hausdorff.
We are interested here in the case where the topology is defined by a seminorm. Such a topology is in particular locally convex.
We record a version of the Hahn-Banach Separation Theorem.

\begin{theorem}\label{hahn banach separation} Suppose $V$ is a topological $\mathbb{R}$-vector space, $A$ and $B$ are non-empty disjoint convex subsets of $V$ and $A$ is open in $V$. Then there exists a continuous linear map $L : V \rightarrow \mathbb{R}$ and $t \in \mathbb{R}$ such that $L(a)<t \le L(b)$ for all $a\in A$ and for all $b\in B$. If $B$ is a cone we can choose $t=0$.
\end{theorem}

\begin{proof}
See \cite[Theorem 7.3.2]{Jar} for the proof of the first assertion, \cite[Theorem 2.4]{GKS} for the proof of the second assertion.
\end{proof}

\begin{cor}\label{hahn banach corollary} Suppose $V$ is a locally convex topological $\mathbb{R}$-vector space and $C$ is a cone in $V$. The closure of $C$ in $V$ consists of all $v\in V$ satisfying $L(v)\ge 0$ for all continuous linear maps $L : V \rightarrow \mathbb{R}$ such that $L\ge 0$ on $C$.
\end{cor}

\begin{proof} Immediate from Theorem \ref{hahn banach separation}.
\end{proof}

\section{polynomial case}\label{main case}

Throughout, $\mathbb{N} := \{ 0,1, \dots\}$, $n$ denotes a fixed positive integer, $\underline{X}$ denotes the $n$-tuple of variables $X_1,\dots,X_n$, and $\underline{X}^s := X_1^{s_1}\dots X_n^{s_n}$, for $s = (s_1,\dots,s_n) \in \mathbb{N}^n$.

For any function $\phi : \mathbb{N}^n \rightarrow \mathbb{R}_{\ge 0}$, we define $$\mathcal{K}_{\phi} := \{ \underline{x}\in \mathbb{R}^n \mid |\underline{x}^s| \le \phi(s) \ \forall \ s\in \mathbb{N}^n \}.$$
Fix an integer $d\ge 1$. We denote by $M_{\phi,2d}$ the $\sum \mathbb{R}[\underline{X}]^{2d}$-module of $\mathbb{R}[\underline{X}]$ generated by the elements $\phi(s)\pm \underline{X}^s$, $s \in \mathbb{N}^n$.
$M_{\phi,2d}$ is archimedean. This is a consequence of the fact that $$\sum_s |f_s|\phi(s)+ f =\sum_{f_s>0} |f_s|(\phi(s)+\underline{X}^s)+\sum_{f_s<0} |f_s|(\phi(s)-\underline{X}^s)\in M_{\phi,2d},$$ for any $f = \sum_s f_s\underline{X}^s \in \mathbb{R}[\underline{X}]$.
Also, $\mathcal{K}_{\phi}$ is the non-negativity set of $M_{\phi,2d}$ in $\mathbb{R}^n$ so, by Jacobi's theorem, any $f \in \mathbb{R}[\underline{X}]$ strictly positive on $\mathcal{K}_{\phi}$ belongs to $M_{\phi,2d}$.\footnote{If one insists on $\mathcal{K}_{\phi} \ne \emptyset$ (equivalently, $-1 \notin M_{\phi,2d}$) it is necessary to assume $\phi(0) \ge 1$.}

\begin{definition} A function $\phi : \mathbb{N}^n \rightarrow \mathbb{R}_{\ge 0}$ is called an \it absolute value \rm if
\begin{enumerate}
\item $\phi(0)\ge 1$,
\item $\phi(s+t) \le \phi(s)\phi(t)$ $\forall$  $s,t \in \mathbb{N}^n$.
%$$\phi(s+t) \le \phi(s)\phi(t) \ \forall \ s,t \in \mathbb{N}^n.$$
\end{enumerate}
\end{definition}

Suppose now that $\phi$ is an absolute value.
Denote by $\mathbb{R}[[\underline{X}]]$ the ring of formal power series in $X_1,\dots,X_n$ with coefficients in $\mathbb{R}$. For $f = \sum_s  f_s \underline{X}^s \in \mathbb{R}[[\underline{X}]]$ define the \it $\phi$-seminorm \rm of $f$ to be $\| f \|_{\phi} := \sum_s |f_s|\phi(s)$ and denote by $\mathbb{R}[[\underline{X}]]_{\phi}$ the subset of $\mathbb{R}[[\underline{X}]]$ consisting of all $f \in \mathbb{R}[[\underline{X}]]$ having finite $\phi$-seminorm.
Using
$$\| f+g\|_{\phi} \le \| f\|_{\phi} +\| g\|_{\phi}, \ \|rf\|_{\phi} = |r|\|f\|_{\phi} \text{ and }
\| fg\|_{\phi} \le \| f\|_{\phi} \| g\|_{\phi}$$ (these are easily verified),
we see that $\mathbb{R}[[\underline{X}]]_{\phi}$ is a subalgebra of the $\mathbb{R}$-algebra $\mathbb{R}[[\underline{X}]]$. It is the closure of $\mathbb{R}[\underline{X}]$ in the topology induced by the $\phi$-seminorm.

\begin{lemma} \label{sqrt} Suppose  $r \in \mathbb{R}$, $s \in \mathbb{N}^n$,
$r>\phi(s)$. Then $(r\pm \underline{X}^s)^{1/2d} \in \mathbb{R}[[\underline{X}]]_{\phi}$.
\end{lemma}

\begin{proof} We may assume $s\ne 0$.
Denote by $\sum_{i=0}^{\infty} a_it^i$ the power series expansion of $f(t) = (r\pm t)^{1/2d}$ about $t=0$, i.e., $a_i = \frac{f^{(i)}(0)}{i!}$.
This has radius of convergence $r$ so it converges absolutely for $|t|<r$. In particular, it converges absolutely for $t=\phi(s)$, i.e., $\sum_{i=0}^{\infty} |a_i|\phi(s)^i <\infty$. Since $\phi(is)\le \phi(s)^i$ for $i\ge 1$, this implies $\sum_{i=0}^{\infty} |a_i|\phi(is) <\infty$, i.e., $(r\pm \underline{X}^s)^{1/2d} = \sum_{i=0}^{\infty} a_i\underline{X}^{is} \in \mathbb{R}[[\underline{X}]]_{\phi}$.
\end{proof}

An important example of an absolute value, perhaps the most important one, is the constant function $1$. If $\phi=1$ then $\mathcal{K}_{\phi} = [-1,1]^n$ and the $\phi$-seminorm is the standard $\ell_1$-norm $\| f \|_1 := \sum_s |f_s|$.

\begin{theorem}\label{new} Suppose $\phi$ is an absolute value on $\mathbb{N}^n$ %satisfying $\phi(0)\ge 1$
and $f\in \mathbb{R}[\underline{X}]$, $f>0$ on $\mathcal{K}_{\phi}$. Then $f \in \sum \mathbb{R}[[\underline{X}]]_{\phi}^{2d}$.
\end{theorem}

\begin{proof} For each real $\delta>0$ consider the function $\phi+\delta : \mathbb{N}^n \rightarrow \mathbb{R}_{\ge 0}$ defined by $$(\phi+\delta)(s) := \phi(s)+\delta.$$ Since $\cap_{\delta>0} \mathcal{K}_{\phi+\delta} = \mathcal{K}_{\phi}$, each $\mathcal{K}_{\phi+\delta}$ is compact and $f>0$ on $\mathcal{K}_{\phi}$, $\exists$ $\delta >0$ such that $f >0$ on $\mathcal{K}_{\phi+\delta}$.
The $\sum \mathbb{R}[\underline{X}]^{2d}$-module $M_{\phi+\delta,2d}$ of $\mathbb{R}[\underline{X}]$ generated by the elements $\phi(s)+\delta \pm \underline{X}^s$, $s \in \mathbb{N}^n$ is archimedean. By Jacobi's theorem, $f \in M_{\phi+\delta,2d}$. By Lemma \ref{sqrt}, $(\phi(s)+\delta \pm \underline{X}^s)^{1/2d} \in \mathbb{R}[[\underline{X}]]_{\phi}$ for each $s\in \mathbb{N}^n$.
\end{proof}

\begin{cor}\label{density} For any absolute value $\phi$ on $\mathbb{N}^n$ the closure of the cone $\sum \mathbb{R}[\underline{X}]^{2d}$ in $\mathbb{R}[\underline{X}]$ in the topology induced by the $\phi$-seminorm is
$\operatorname{Pos}(\mathcal{K}_{\phi})$.
\end{cor}

\begin{proof} 
The inclusion ($\subseteq$) follows from continuity of
the evaluation map $f \mapsto f(\underline{x})$, for $\underline{x}\in \mathcal{K}_{\phi}$, which follows in turn from the fact that $|f(\underline{x})-g(\underline{x})| \le \| f-g\|_{\phi}$, for $\underline{x} \in \mathcal{K}_{\phi}$. To prove ($\supseteq$), suppose $f \in \mathbb{R}[\underline{X}]$, $f\ge 0$ on $\mathcal{K}_{\phi}$ and $\epsilon>0$. Then $f+\frac{\epsilon}{2} > 0$ on $\mathcal{K}_{\phi}$ so $\exists$ $f_1,\dots,f_m \in \mathbb{R}[[\underline{X}]]_{\phi}$ such that $f+\frac{\epsilon}{2} = f_1^{2d}+\dots+f_m^{2d}$, by Theorem \ref{new}. Take $g = g_1^{2d}+\dots+g_m^{2d}$ where $g_i \in \mathbb{R}[\underline{X}]$ is such that $\| f_i^{2d}-g_i^{2d}\|_{\phi} \le \frac{\epsilon}{2m}$, $i=1,\dots,m$. Then $g\in \sum \mathbb{R}[\underline{X}]^{2d}$, $\| f-g\|_{\phi}\le \epsilon$.
\end{proof}

\begin{cor}\label{berg et al} Suppose $L : \mathbb{R}[\underline{X}] \rightarrow \mathbb{R}$ is linear, $L(p^{2d}) \ge 0$ $\forall$ $p\in \mathbb{R}[\underline{X}]$, and $\exists$ an absolute value $\phi$ and a constant $C>0$ such that  $|L(\underline{X}^s)| \le C\phi(s)$ $\forall$ $s\in \mathbb{N}^n$. Then $\exists$ a unique positive borel measure $\mu$ on the set $\mathcal{K}_{\phi}$
such that $L(f) = \int f \, d\mu$ $\forall$ $f\in \mathbb{R}[\underline{X}]$.
\end{cor}

\begin{proof} The hypothesis implies that $|L(f)-L(g)| \le C\| f-g\|_{\phi}$, so $L$ is continuous.
Fix $f \in \operatorname{Pos}(\mathcal{K}_{\phi})$. Fix $\epsilon >0$. By Corollary \ref{density}, $\exists$ $g\in \sum \mathbb{R}[\underline{X}]^{2d}$ such that $\| f-g\|_{\phi} \le \epsilon$, so $| L(f)-L(g)|\le C \epsilon$. Since $L(g)\ge 0$, this implies $L(f)\ge -C\epsilon$.
Since $\epsilon>0$ is arbitrary, this implies $L(f)\ge 0$. The conclusion follows, by Theorem \ref{haviland,etc}(2).
\end{proof}

\begin{rem} (i) In the case $d=1$ Corollary \ref{berg et al} is well-known. It can be obtained by applying \cite[Theorem 4.2.5]{BCR} to the semigroup $(\mathbb{N}^n,+)$ equipped with the identity involution;  see \cite[Theorem 2.2]{LN}. At the same time, the proof given here is new, even in the case $d=1$.

(ii) The converse of Corollary \ref{berg et al} holds: If $L(f) = \int f \, d\mu$ where $\mu$ is a positive borel measure on $\mathcal{K_{\phi}}$ then $L(p^{2d})\ge 0$ for all $p \in \mathbb{R}[\underline{X}]$ and $|L(\underline{X}^s)| \le C\phi(s)$ where $C:= \mu(\mathcal{K}_{\phi})$. This is clear.

(iii) We have proved Corollary \ref{berg et al} from Corollary \ref{density} using Theorem \ref{haviland,etc}(2). One can also prove Corollary \ref{density} from Corollary \ref{berg et al} using Corollary \ref{hahn banach corollary}. In this way, Corollary \ref{density} and Corollary \ref{berg et al} can be seen to carry exactly the same information.

(iv) Corollary \ref{density} extends \cite[Theorem 9.1]{BC}.

(v) In \cite{LN}, Lasserre and Netzer use \cite[Theorem 4.2.5]{BCR} to prove that for $\phi$ equal to the constant function 1 and for any $f\in \operatorname{Pos}(\mathcal{K}_{\phi})$ and any real $\epsilon >0$, and any integer $k\ge 1$ sufficiently large (depending on $\epsilon$ and $f$), $$f+\epsilon(1+\sum_{i=1}^n X_i^{2k}) \in \sum \mathbb{R}[\underline{X}]^2.$$ It is not clear how to extend this result with $\sum \mathbb{R}[\underline{X}]^2$ replaced by $\sum \mathbb{R}[\underline{X}]^{2d}$.

(vi) In \cite[Theorem 4.6]{GKS} and \cite[Theorem 4.10]{GKS} Ghasemi, Kuhlmann and Samei prove analogs of \cite[Theorem 9.1]{BC} for the $\ell_{p}$-norms $$\| f \|_{p} := (\sum_{s\in \mathbb{N}^n} |f_s|^p)^{1/p}, \, 1\le p < \infty, \, \| f\|_{\infty} := \sup \{ |f_s| \mid s\in \mathbb{N}^n\}$$ and for certain weighted versions of the $\ell_p$-norms.
Replacing \cite[Theorem 9.1]{BC} by
Corollary \ref{density} in these proofs, one verifies that these results carry over word-for-word with $\sum \mathbb{R}[\underline{X}]^2$ replaced by $\sum \mathbb{R}[\underline{X}]^{2d}$.
\end{rem}

\section{general case}
\label{extension}

Our goal in this section is to extend Corollary \ref{density} and Corollary \ref{berg et al} to arbitrary commutative semigroups with involution; see Theorem \ref{density 2} and Corollary \ref{berg et al 2}.

As in \cite{BCR} \cite{BM}, we work with a commutative $*$-semigroup $S=(S,\cdot, 1,*)$ with neutral element $1$ and involution $*$. The involution $* :S \rightarrow S$ satisfies $$(st)^*=s^*t^*, \ (s^*)^* = s \text{ and } 1^*=1.$$
We denote by $\mathbb{C}[S]$ the semigroup ring of $S$ with coefficients in $\mathbb{C}$. Elements of $\mathbb{C}[S]$ have the form $f= \sum_{s\in S} f_ss$ (finite sum), $f_s\in \mathbb{C}$. $\mathbb{C}[S]$ has the structure of a $\mathbb{C}$-algebra with involution. Addition, scalar multiplication and multiplication are defined by $$f+g =\sum (f_s+g_s)s, \ zf = \sum (zf_s)s, \ fg = \sum_{s,t} f_sg_tst = \sum_u (\sum_{st=u} f_sg_t)u.$$
The involution is defined by $$f^* = \sum \overline{f_s}s^*.$$
An element $f \in \mathbb{C}[S]$ is said to be \it symmetric \rm if $f^* = f$, i.e., if $f_{s^*}=\overline{f_s}$ for all $s\in S$. We denote the $\mathbb{R}$-algebra consisting of all symmetric elements of $\mathbb{C}[S]$ by $A_S$. Clearly $$\mathbb{C}[S] = A_S\oplus iA_S.$$ As an $\mathbb{R}$-vector space $A_S$ is generated by the elements $s+s^*$ and $i(s-s^*)$, $s\in S$. If the involution on $S$ is the identity, i.e., $s^*=s$ for all $s\in S$, then $A_S = \mathbb{R}[S]$, the semigroup ring of $S$ with coefficients in $\mathbb{R}$.

A \it semicharacter \rm of $S$ is a function $\alpha : S \rightarrow \mathbb{C}$ satisfying
\begin{enumerate}
\item $\alpha(1) = 1$;
\item $\alpha(st) = \alpha(s)\alpha(t)$ $\forall$ $s,t \in S$;
\item $\alpha(s^*)=\overline{\alpha(s)}$ $\forall$ $s\in S$.
\end{enumerate}
We denote by $S'$ the set of all semicharacters of $S$.  Semicharacters $\alpha$ of $S$ correspond bijectively to $*$-algebra homomorphisms $\alpha : \mathbb{C}[S] \rightarrow \mathbb{C}$ via $\alpha(f) := \sum_{s\in S} f_s \alpha(s)$. In turn, $*$-algebra homomorphisms $\alpha : \mathbb{C}[S] \rightarrow \mathbb{C}$ correspond bijectively to ring homomorphisms $\alpha: A_S \rightarrow \mathbb{R}$ via $\alpha(f+gi) = \alpha(f)+\alpha(g)i$. In this way, $S'$ and $X_{A_S}$ are naturally identified.

For any function $\phi : S \rightarrow \mathbb{R}_{\ge 0}$ define
$$\mathcal{K}_{\phi} := \{ \alpha \in S' \mid |\alpha(s)| \le \phi(s) \ \forall \ s\in S \}.$$
Fix an integer $d\ge 1$. Denote by $M_{\phi,2d}$ the $\sum A_S^{2d}$-module of $A_S$ generated by the elements $$\phi(s)^2- ss^*, \ 2\phi(s)\pm (s+s^*) \text{ and } 2\phi(s)\pm i(s-s^*), \ s\in S.$$

\begin{lemma} \label{technical} \
\begin{enumerate}
\item $M_{\phi,2d}$ is archimedean.
\item The non-negativity set of $M_{\phi,2d}$ in $S'$ is $\mathcal{K}_{\phi}$.
\end{enumerate}
\end{lemma}

\begin{proof} (1) The elements $s+s^*$, $i(s-s^*)$ generate $A_S$ as an $\mathbb{R}$-vector space and  $2\phi(s)\pm (s+s^*), 2\phi(s)\pm i(s-s^*) \in M_{\phi,2d}$, so  $M_{\phi,2d}$ is archimedean.

(2) For $\alpha \in S'$, $$|\alpha(s)|\le \phi(s) \ \Leftrightarrow \ \alpha(s)\overline{\alpha(s)} \le \phi(s)^2 \ \Leftrightarrow \ \phi(s)^2-ss^* \ge 0 \text{ at } \alpha.$$ Also, using the inequality $\sqrt{a^2+b^2} \ge \max\{ |a|,|b|\}$, $$|\alpha(s)| \le \phi(s) \ \Rightarrow \ |\frac{\alpha(s)+\overline{\alpha(s)}}{2}| \le \phi(s) \ \Leftrightarrow \ 2\phi(s)\pm(s+s^*)\ge 0 \text{ at } \alpha,$$ and $$|\alpha(s)| \le \phi(s) \ \Rightarrow \ |\frac{\alpha(s)-\overline{\alpha(s)}}{2i}| \le \phi(s) \ \Leftrightarrow \ 2\phi(s)\pm i(s-s^*)\ge 0 \text{ at } \alpha.$$
\end{proof}

A function $\phi : S \rightarrow \mathbb{R}_{\ge 0}$ is called an \it absolute value \rm
if
\begin{enumerate}
\item $\phi(1) \ge 1$;
\item $\phi(st) \le \phi(s)\phi(t)$ $\forall$ $s,t \in S$;
\item $\phi(s^*)=\phi(s)$ $\forall$ $s\in S$.
\end{enumerate}
Suppose that $\phi$ is an absolute value on $S$. For $f = \sum_s  f_s s \in \mathbb{C}[S]$ define the \it $\phi$-seminorm \rm of $f$ to be $\| f \|_{\phi} := \sum_s |f_s|\phi(s)$.
One checks easily that
\begin{align*} \| f+g\|_{\phi} \le &\| f\|_{\phi} +\| g\|_{\phi}, \ \|zf\|_{\phi} = |z|\|f\|_{\phi}, \\
\| fg\|_{\phi} \le &\| f\|_{\phi} \| g\|_{\phi} \text{ and } \| f^*\|_{\phi} = \| f \|_{\phi},
\end{align*}
so the addition, scalar multiplication, multiplication and conjugation in the semigroup algebra $\mathbb{C}[S]$ are continuous in the topology induced by the $\phi$-seminorm.

\begin{lemma} \label{basicapproximation}
Let $r\in \mathbb{R}$, $f\in A_S$, $r>\| f\|_{\phi}$. Then, for each real $\epsilon >0$, there exists $g \in A_S$ such that $\| (r+f)-g^{2d}\|_{\phi}<\epsilon$.
\end{lemma}

\begin{proof} Consider the $\mathbb{R}$-algebra homomorphism $\tau : \mathbb{R}[X] \rightarrow A_S$ defined by $X\mapsto f$ and consider the absolute value $\phi'$ on $(\mathbb{N},+)$ defined by $\phi'(i) = \| f^i\|_{\phi}$. Applying Lemma \ref{sqrt} we see that $(r+ X)^{1/2d} \in \mathbb{R}[[X]]_{\phi'}$. Combining this with the density of $\mathbb{R}[X]$ in $\mathbb{R}[[X]]_{\phi'}$ and the continuity of the multiplication in the topology induced by the $\phi'$-seminorm, there exists $h\in \mathbb{R}[X]$ such that $\| r+ X - h^{2d}\|_{\phi'}<\epsilon$. Take $g = \tau(h)$. Since $\tau(r+ X-h^{2d}) = r+f-g^{2d}$ and  $\| \tau(p)\|_{\phi} \le \| p \|_{\phi'}$, for all $p \in \mathbb{R}[X]$, this completes the proof.
\end{proof}

\begin{theorem} \label{density 2} Suppose $\phi$ is an absolute value on a commutative semigroup $S$ with involution and $d$ is any positive integer. Then the closure of the cone $\sum A_S^{2d}$ in $A_S$ in the topology induced by the $\phi$-seminorm is equal to  $\operatorname{Pos}(\mathcal{K}_{\phi})$.
\end{theorem}

\begin{proof} Since $\sum A_S^{2d} \subseteq \operatorname{Pos}(\mathcal{K}_{\phi})$ and $\operatorname{Pos}(\mathcal{K}_{\phi})$ is closed, one inclusion in clear. The fact that $\operatorname{Pos}(\mathcal{K}_{\phi})$ is closed comes from the fact that each $\alpha \in \mathcal{K}_{\phi}$, viewed as a ring homomorphism $\alpha : A_S \rightarrow \mathbb{R}$ in the standard way, satisfies $|\alpha(f)| \le \| f\|_{\phi}$ for all $f\in A_S$, so $\alpha$ is continuous for each $\alpha \in \mathcal{K}_{\phi}$, and $\operatorname{Pos}(\mathcal{K}_{\phi}) = \cap_{\alpha\in \mathcal{K}_{\phi}}\alpha^{-1}(\mathbb{R}_{\ge 0})$.

For the other inclusion we must show if $f\in \operatorname{Pos}(\mathcal{K}_{\phi})$ and $\epsilon >0$, there exists $g \in \sum A_S^{2d}$ such that $\| f-g\|_{\phi} \le \epsilon$. Note that $f+\frac{\epsilon}{2}$ is strictly positive at each $\alpha \in \mathcal{K}_{\phi}$ so, by Lemma \ref{technical} and Jacobi's theorem, $$f+\frac{\epsilon}{2} = \sum_{i=0}^k g_im_i$$ where $g_i \in \sum A_S^{2d}$, $i=0,\dots, k$, $m_0=1$, and $m_i \in \{ \phi(s)^2-ss^*, \ 2\phi(s)\pm (s+s^*), \ 2\phi(s)\pm i(s-s^*)\mid s\in S \}$, $i=1,\dots,k$. Choose $\delta>0$ so that $(\sum_{i=1}^k \| g_i\|_{\phi}) \delta \le \frac{\epsilon}{2}$. By Lemma, \ref{basicapproximation} there exists $h_i \in A_S$ such that $\| \frac{\delta}{2}+m_i-h_i^{2d}\|_{\phi} \le \frac{\delta}{2}$, i.e., $\| m_i-h_i^{2d}\|_{\phi} \le \delta$, $i=1,\dots,k$. Take $g= g_0+\sum_{i=1}^k g_ih_i^{2d}$. Then $g\in \sum A_S^{2d}$, and $$\| f-g\|_{\phi} = \| \sum_{i=1}^k g_im_i-\sum_{i=1}^k g_ih_i^{2d}-\frac{\epsilon}{2}\| \le  \sum_{i=1}^k \|g_i\|_{\phi}\|m_i-h_i^{2d}\|_{\phi}+\frac{\epsilon}{2}\le \epsilon.$$
\end{proof}

\begin{cor} \label{berg et al 2} Let $S$ be a commutative semigroup with involution and let $d$ be a positive integer. Let $L: \mathbb{C}[S] \rightarrow \mathbb{C}$ be a $*$-linear mapping such that $L(p^{2d})\ge 0$ for all $p \in A_S$ and suppose there exists an absolute value $\phi$ on $S$ and a constant $C>0$ such that $|L(s)|\le C\phi(s)$ for all $s\in S$. Then there exists a unique positive borel measure $\mu$ on $\mathcal{K}_{\phi}$ such that $L(f) = \int \hat{f} d\mu$ for each $f\in \mathbb{C}[S]$.
\end{cor}

Here, $\hat{f} : S' \rightarrow \mathbb{C}$ is defined by $\hat{f}(\alpha) := \alpha(f)$ $\forall$ $\alpha \in S'$, equivalently, if $f = g+ih$, $g,h \in A_S$, then $\hat{f} := \hat{g}+i\hat{h}$.

\begin{proof} $*$-linear mappings $L : \mathbb{C}[S] \rightarrow \mathbb{C}$ correspond bijectively to $\mathbb{R}$-linear mappings $L : A_S \rightarrow \mathbb{R}$, the correspondence being given by $L(f+gi) = L(f)+L(g)i$. The hypothesis implies that $|L(f)-L(g)| \le C\| f-g\|_{\phi}$, so $L$ is continuous.
Fix $f \in \operatorname{Pos}(\mathcal{K}_{\phi})$. Fix $\epsilon >0$. By Theorem \ref{density 2}, $\exists$ $g\in \sum A_S^{2d}$ such that $\| f-g\|_{\phi} \le \epsilon$, so $| L(f)-L(g)|\le C \epsilon$. Since $L(g)\ge 0$, this implies $L(f)\ge -C\epsilon$.
Since $\epsilon>0$ is arbitrary, this implies $L(f)\ge 0$. The conclusion follows, by Theorem \ref{haviland,etc}(2).
\end{proof}

\begin{rem}\label{no name} For $p\in \mathbb{C}[S]$, $p=q+ir$, $q,r \in A_S$, $pp^* = (q+ir)(q-ir) = q^2+r^2$. Thus, for $L : \mathbb{C}[S] \rightarrow \mathbb{C}$ $*$-linear, $L(p^2)\ge 0$ $\forall$ $p \in A_S$ $\Leftrightarrow$ $L(pp^*)\ge 0$ $\forall$ $p\in \mathbb{C}[S]$  $\Leftrightarrow$ $L$ is positive (semi)definite, terminology as in \cite{BC}, \cite{BCR} and \cite{BM}.
Consequently, Corollary \ref{berg et al 2} generalizes and provides another proof of what is proved in \cite[Corollary 2.5]{BC} and \cite[Theorem 4.2.5]{BCR}.
\end{rem}

\section{Berg-Maserick result} \label{weak absolute values}

In this section we relax the requirement that an absolute value satisfies $\phi(1)\ge 1$. If $\phi(1)<1$ then, since $\phi(s) = \phi(s1)\le \phi(s)\phi(1)$ $\forall$ $s\in S$, $\phi$ is identically zero. Then $\| \, \|_{\phi}$ is also identically zero, so the topology on $\mathbb{C}[S]$ is the trivial one and the closure of $\sum A_S^{2d}$ in $A_S$ is $A_S$. At the same time, $\mathcal{K}_{\phi} = \emptyset$ so $\operatorname{Pos}(\mathcal{K}_{\phi}) = A_S$. Consequently, Theorem \ref{density 2} and Corollary \ref{berg et al 2} continue to hold in this more general situation.

We explain how the Berg-Maserick result \cite[Theorem 2.1]{BM} can be deduced as a consequence of Corollary \ref{berg et al 2}. See Corollary \ref{berg-maserick}.

A \it weak absolute value \rm on $S$ is a function $\phi : S \rightarrow \mathbb{R}_{\ge 0}$ satisfying $$\phi(ss^*) \le \phi(s)^2 \ \forall \ s\in S.$$
Replacing $s$ by $s^*$, we see that $\phi(ss^*) \le \phi(s^*)^2$, so $$\phi(ss^*)\le \min\{ \phi(s)^2,\phi(s^*)^2\} \ \forall \ s\in S,$$ for any weak absolute value $\phi$ on $S$.

For any weak absolute value $\phi$ on $S$, define $\phi' : S \rightarrow \mathbb{R}_{\ge 0}$ by
$$\phi'(s) = \inf \{ \prod_{i=1}^k \min\{ \phi(s_i),\phi(s_i^*)\} \mid \ k\ge 1, \ s_1,\dots,s_k\in S, \ s=s_1\dots s_k\}.$$

\begin{lemma} \label{preabsolute value} Let $\phi$ be a weak absolute value on $S$. Then
\begin{enumerate}
\item $\phi'$ is an absolute value (possibly $\phi' \equiv 0$).
\item If $L: \mathbb{C}[S]\rightarrow \mathbb{C}$ is $*$-linear and positive semidefinite and $\exists$ $C>0$ such that $|L(s)| \le C\phi(s)$ $\forall$ $s\in S$, then $|L(s)| \le C\phi'(s)$ $\forall$ $s\in S$.
\item $\mathcal{K}_{\phi} = \mathcal{K}_{\phi'}$.
\end{enumerate}
\end{lemma}

\begin{proof} (1) This is clear. 
(2) It suffices to show $$|L(s_1\dots s_k)| \le C \prod_{i=1}^k \min\{ \phi(s_i),\phi(s_i^*)\} \ \forall \ s_1,\dots,s_k \in S.$$ Since $|L(s)| \le C\phi(s)$ and $|L(s)| = |\overline{L(s)}| = |L(s^*)| \le C\phi(s^*)$, the result is clear when $k=1$. Suppose now that $k\ge 2$. We make use of the Cauchy-Schwartz inequality for the inner product $$\langle f,g\rangle := L(fg^*), \ f,g \in \mathbb{C}[S].$$ This implies, in particular, that $$|L(st^*)|^2 \le L(ss^*)L(tt^*) \ \forall \ s,t \in S.$$ Using this we obtain
\begin{align*}
|L(s_1\dots s_k)|^2 \le &L(s_1s_1^*)L(s_2s_2^*\dots s_ks_k^*) \le C\phi(s_1s_1^*)C\prod_{i=2}^k \phi(s_is_i^*)\\= &C^2\prod_{i=1}^k \phi(s_is_i^*) \le C^2 \prod_{i=1}^k \min\{ \phi(s_i)^2,\phi(s_i^*)^2\}.
\end{align*}
(the second inequality by induction on $k$). The result follows, by taking square roots. (3). Since $\phi'(s) \le \phi(s)$ for all $s\in S$, the inclusion $\mathcal{K}_{\phi'} \subseteq \mathcal{K}_{\phi}$ is clear. For the other inclusion, note that each $\alpha \in S'$ is positive semidefinite, so
$\mathcal{K}_{\phi} \subseteq \mathcal{K}_{\phi'}$, by (2).
\end{proof}

\begin{cor} \label{second corollary} Suppose $\phi$ is a weak absolute value on $S$. Then the closure of $\sum A_S^2$ in $A_S$ in the topology induced by the $\phi$-seminorm $\| f\|_{\phi} := \sum |f_s|\phi(s)$ is equal to $\operatorname{Pos}(\mathcal{K}_{\phi})$.
\end{cor}

\begin{proof} Denote the closure of $\sum A_S^2$ in $A_S$ in the topology induced by the $\phi$-seminorm by $\overline{\sum A_S^2}^{\phi}$. By Lemma \ref{preabsolute value}(2), an $\mathbb{R}$-linear map $L : A_S \rightarrow \mathbb{R}$ non-negative on $\sum A_S^2$ is continuous in the topology induced by $\| \, \|_{\phi}$ iff it is continuous in the topology induced by $\| \, \|_{\phi'}$. It follows, using Corollary \ref{hahn banach corollary}, that $\overline{\sum A_S^2}^{\phi} = \overline{\sum A_S^2}^{\phi'}$. By Lemma \ref{preabsolute value}(3), $\mathcal{K}_{\phi} = \mathcal{K}_{\phi'}$ so $\operatorname{Pos}(\mathcal{K}_{\phi}) = \operatorname{Pos}(\mathcal{K}_{\phi'})$. The result follows now using Lemma \ref{preabsolute value}(1) and Theorem \ref{density 2}. 
\end{proof}

\begin{cor} \label{berg-maserick} Suppose $L : \mathbb{C}[S]\rightarrow \mathbb{C}$ is $*$-linear and positive semidefinite and there exists a weak absolute value $\phi$ on $S$ and a constant $C>0$ such that $|L(s)| \le C\phi(s)$ $\forall$ $s\in S$. Then there exists a unique positive borel measure $\mu$ on $\mathcal{K}_{\phi}$ such that $L(f) = \int \hat{f} d\mu$ for each $f\in \mathbb{C}[S]$.
\end{cor}

\begin{proof} 
In view of Lemma \ref{preabsolute value}, this is immediate from Corollary \ref{berg et al 2}.
\end{proof}

Since the argument in Lemma \ref{preabsolute value}(2) makes essential use of the Cauchy-Schwartz inequality, it seems unlikely that Corollaries \ref{second corollary} and \ref{berg-maserick} extend to the case $d>1$.

\section{appendix}\label{appendix}

Let $A$ be a commutative ring with $1$. For simplicity assume $\mathbb{Q} \subseteq A$. In what follows, $P$ is assumed to be a preprime of $A$, and $M \subseteq A$ is a $P$-module.

We explain how the simple proof of Jacobi's theorem in the case $d=1$ found in \cite[Theorem 5.4.4]{M} can be extended to $d>1$. We use the following lemma:

\begin{lemma} \label{basic trick} Suppose $M$ is archimedean, $a\in A$, $t\in P$, $at-1\in M$. Let $n$ be a positive integer which is even. Then for any sufficiently large $k \in \mathbb{Q}$, $$k^n(a+r)-1-(k-t)^n(a+r) \in M$$ for each non-negative $r \in \mathbb{Q}$.
\end{lemma}

\begin{proof} Since
\begin{align*}
k^n(a+r)-1-(k-t)^n(a+r) =& -1 -\sum_{i=1}^n \begin{pmatrix} n\\i\end{pmatrix}k^{n-i}(-t)^i(a+r) \\ =&-1+\begin{pmatrix}n\\1\end{pmatrix}k^{n-1}-\sum_{i=2}^n \begin{pmatrix} n\\i\end{pmatrix} k^{n-i}(-t)^ia \\+&\begin{pmatrix}n\\1\end{pmatrix} k^{n-1}(at-1)+rt\sum_{i=1}^n \begin{pmatrix}n\\i\end{pmatrix} k^{n-i}(-t)^{i-1}
\end{align*}
and $at-1\in M$, it suffices to show that
\begin{equation}\label{eq1}
\sum\limits_{i=1}^n \begin{pmatrix}n\\i\end{pmatrix} k^{n-i}(-t)^{i-1} \in M
\end{equation}
and
\begin{equation}\label{eq2}
-1+\begin{pmatrix}n\\1\end{pmatrix}k^{n-1}-\sum\limits_{i=2}^n \begin{pmatrix} n\\i\end{pmatrix} k^{n-i}(-t)^ia \in M
\end{equation}
for sufficiently large $k$. Since $$\sum_{i=1}^n \begin{pmatrix} n\\i\end{pmatrix} k^{n-i}(-t)^{i-1} = k^{n-2}[\begin{pmatrix}n\\1\end{pmatrix} k -\begin{pmatrix} n\\2\end{pmatrix}t]+k^{n-4}t^2[\begin{pmatrix}n\\3\end{pmatrix} k-\begin{pmatrix} n\\4\end{pmatrix}t]+\dots,$$ and $M$ is archimedean, so $\begin{pmatrix}n\\1\end{pmatrix} k -\begin{pmatrix} n\\2\end{pmatrix}t, \begin{pmatrix}n\\3\end{pmatrix} k-\begin{pmatrix} n\\4\end{pmatrix}t, \dots$ belong to $M$ for $k$ sufficiently large, (\ref{eq1}) is clear, for $k$ sufficiently large.

Regarding (\ref{eq2}), write

\begin{align*}
\begin{pmatrix} n\\3\end{pmatrix} k^{n-3}t^3a=&\begin{pmatrix}n\\3\end{pmatrix}k^{n-3}t^2(k+ta)-\begin{pmatrix}n\\3\end{pmatrix}k^{n-2}t^2,\\ -\begin{pmatrix}n\\4\end{pmatrix} k^{n-4}t^4a =& \begin{pmatrix}n\\4\end{pmatrix}k^{n-4}t^2(k^2-t^2a)-\begin{pmatrix}n\\4\end{pmatrix} k^{n-2}t^2, \text{ etc}.,
\end{align*}
choosing $k$ so large that $k+ta$, $k^2-t^2a$, etc., belong to $M$,  we are reduced to showing that
\begin{equation}\label{eq3}
-1+\begin{pmatrix}n\\1\end{pmatrix}k^{n-1}-\begin{pmatrix}n\\2\end{pmatrix}k^{n-2}t^2a-\begin{pmatrix}n\\3
\end{pmatrix} k^{n-2}t^2-\begin{pmatrix}n\\4\end{pmatrix}k^{n-2}t^2-\dots \in M
\end{equation}
for $k$ sufficiently large. Dividing through by $k^{n-2}$, using the fact that $M$ is archimedean and $\begin{pmatrix}n\\ 1\end{pmatrix} k-\frac{1}{k^{n-2}} \ge \begin{pmatrix}n\\ 1\end{pmatrix}k-1$ if $k\ge 1$, we see that this is true, i.e., (\ref{eq3}) does indeed hold for $k$ sufficiently large.
\end{proof}

\noindent
\it Proof of Theorem \ref{jacobi theorem}. \rm We argue as in \cite[Theorem 5.4.4]{M}. Let $P= \sum A^{2d}$. Set $M_1 := M-aP$. Since $M\subseteq M_1$, $M_1$ is archimedean. The assumption $\alpha(a)>0$ for all $\alpha \in \mathcal{K}_M$ implies $\mathcal{K}_{M_1} = \emptyset$ and, as noted earlier, $P$ is generating, so, by \cite[Corollary 5.4.1]{M}, $-1 \in M_1$. Thus $-1 = s-at$, $s\in M$, $t\in P$, so $at-1 = s\in M$. Apply Lemma \ref{basic trick} with $n=2d$ to conclude that $$k^{2d}(a+r)-1-(k-t)^{2d}(a+r) \in M$$ for each $r \in \mathbb{Q}_{\ge 0}$, for any $k$ sufficiently large. Dividing by $k^{2d}$, we see that $$a+r\in M \ \Rightarrow \ a+r-\frac{1}{k^{2d}} \in M.$$ Iterating, we obtain eventually that $a+r \in M$ for some negative $r\in \mathbb{Q}$, so $a = (a+r)+(-r)\in M$.\qed

\medskip

Recall: A preprime $P$ of $A$ is \it torsion \rm if $\forall$ $a\in A$ $\exists$ $n\ge 1$ such that $a^n \in P$, and $P$ is \it weakly torsion \rm if $\forall$ $a\in A$ $\exists$ rational $r>0$ and $n\ge 1$ such that $(r+a)^n \in P$. Clearly, for any preprime $P$, $$P \text{ is a preordering } \Rightarrow \ P \text{ is torsion } \Rightarrow \ P \text{ is weakly torsion.}$$ It is proved in \cite[Lemma 2.4]{M0} that any preprime which is weakly torsion is generating. Moreover, one has the following extension of Jacobi's result:

\begin{theorem} \label{general result} Suppose $M\subseteq A$ is an archimedean $P$-module, $P$ a weakly torsion preprime of $A$. Then, for any $a\in A$, $$\hat{a} >0 \text{ on } \mathcal{K}_M \ \Rightarrow \ a\in M.$$
\end{theorem}

\begin{proof} See \cite[Theorem 2.3]{M0}.
\end{proof}
Unfortunately, it is not clear how  Lemma \ref{basic trick} can be applied to give a proof of Theorem \ref{general result}.

\end{document}